\documentclass{amsart}
\usepackage{amsmath,amssymb, amscd}

\DeclareFontEncoding{OT2}{}{} 
  \newcommand{\textcyr}[1]{%
    {\fontencoding{OT2}\fontfamily{wncyr}\fontseries{m}\fontshape{n}%
     \selectfont #1}}
\newcommand{\Sha}{{\mbox{\textcyr{Sh}}}}

\newtheorem{lemma}{Lemma}
\newtheorem{prop}[lemma]{Proposition}
\newtheorem{cor}[lemma]{Corollary}
\newtheorem{thm}[lemma]{Theorem}
\newtheorem{example}[lemma]{Example}
\newtheorem{thm?}{Theorem?}

\newtheorem*{mainthm}{Main Theorem}

\newtheorem{remark}[lemma]{Remark}

\newcommand{\pp}{\mathfrak{p}}
\newcommand{\mm}{\mathfrak{m}}
\renewcommand{\gg}{\mathfrak{g}}

\newcommand{\F}{\ensuremath{\mathbb F}}
\newcommand{\Fp}{\ensuremath{\F_p}}

\newcommand{\PP}{\mathbb{P}}

\newcommand{\N}{\ensuremath{\mathbb N}}
\newcommand{\Q}{\ensuremath{\mathbb Q}}

\newcommand{\R}{\ensuremath{\mathbb R}}
\newcommand{\Z}{\ensuremath{\mathbb Z}}

\newcommand{\ra}{\ensuremath{\rightarrow}}

\newcommand{\lra}{\ensuremath{\longrightarrow}}

\newcommand{\Hom}{\operatorname{Hom}}

\newcommand{\Ker}{\operatorname{Ker}}

\newcommand{\Gal}{\operatorname{Gal}}
\newcommand{\Aut}{\operatorname{Aut}}
\newcommand{\tors}{\operatorname{tors}}

\newcommand{\sep}{\operatorname{sep}}

\newcommand{\Spec}{\operatorname{Spec}}

\newcommand{\gl}{\mathfrak{g}_l}

\newcommand{\Pic}{\operatorname{Pic}}

\newcommand{\gk}{\mathfrak{g}_k}
\newcommand{\car}{\operatorname{char}}

\newcommand{\ns}{\operatorname{\ns}}

\newcommand{\rank}{\operatorname{rank}}
\renewcommand{\gk}{\mathfrak{g}_K}
\renewcommand{\gl}{\mathfrak{g}_L}
\newcommand{\Res}{\operatorname{Res}}
\newcommand{\Sel}{Sel}
\newcommand{\G}{G}

\def\c#1{\text{$\mathcal{#1}$}}                                     
\def\f#1{\text{$\mathfrak{#1}$}}                                    
\newcommand{\hra}{\hookrightarrow} 

\renewcommand{\pmod}[1]{(\mbox{mod } #1 )}

\newcommand{\NS}{NS}
\renewcommand{\ns}{\operatorname{ns}}

\begin{document}

\title[On the index of genus one curve over IFG fields]{There are genus one curves of every index over every infinite, finitely generated field}

\author{Pete L. Clark}
\email{plclark@gmail.com}

\author{Allan Lacy}
\email{alacy@math.uga.edu}

\thanks{$\copyright$ Pete L. Clark and Allan Lacy, 2013}

\begin{abstract}
Every nontrivial abelian variety over a Hilbertian field in which the weak Mordell-Weil theorem holds admits infinitely many torsors with period any $n > 1$ which is not divisible by the characteristic.  The corresponding statement with ``period'' replaced by ``index'' is plausible but much more challenging.  We show that for every infinite, finitely generated field $K$, there is an elliptic curve $E_{/K}$ which admits infinitely many torsors with index any $n > 1$.
\end{abstract}

\maketitle

\tableofcontents

\section{Introduction}

\subsection{Review of the Period-Index Problem}
\noindent
Let $K$ be a field, with a choice of separable closure $K^{\sep}$ and algebraic closure $\overline{K}$.  Let $\mathfrak{g}_K = \Aut(K^{\sep}/K) = \Aut(\overline{K}/K)$ be the absolute Galois group of $K$.  
\\ \\
By a \textbf{variety} $X_{/K}$, we will mean a finite type integral scheme over $\Spec K$ such that $K$ is algebraically closed in $K(X)$.  
(Thus $X_{/\overline{K}}$ need not be a variety.)  A field extension $L/K$ is a \textbf{splitting extension} for $X$ if $X(L) \neq \varnothing$.  For a regular geometrically integral variety $X_{/K}$, we define the \textbf{index} $I(X)$ as the least positive degree of a $K$-rational zero-cycle on $X$. Equivalently, it is the gcd of all degrees of finite splitting extensions.  Clearly we have $I(X_{/L}) \mid I(X)$ for every extension $L/K$.
\\ \\
Let $M$ be a commutative $\gk$-module. For $i \in \N$ we have the Galois cohomology groups $H^i(K,M) = H^i(\gk,M)$. For $i \geq 1$, $H^i(K,M)$ is a torsion commutative group. For $i \geq 1$ and $\eta \in H^i(K,M)$, the \textbf{period} $P(\eta)$ as the order of $\eta$ in $H^i(K,M)$.  
\\ \\
If $L/K$ is algebraic, then $\mathfrak{g}_L = \Aut(K^{\sep}/L)$ is a closed subgroup of $\mathfrak{g}_K$ and there is a natural \textbf{restriction map}
\[ \Res_L: H^i(K,M) \ra H^i(L,M). \]
We shall also want to consider restriction maps associated to transcendental field extensions, and for this we need a bit more structure: let $A_{/K}$ be a locally finite type commutative group scheme. Then $M = A(K^{\sep})$ is a $\gk$-module, and we write $H^i(K,A)$ for $H^i(K,A(K^{\sep}))$. Let $L/K$ be any field extension. There is a field embedding $\iota: \overline{K} \hookrightarrow \overline{L}$. Any automorphism $\sigma \in \Aut(\overline{L}/L)$ fixes $K$ pointwise so restricts to an automorphism of $\overline{K}$. This gives a continuous group homomorphism $\gl \ra \gk$ and thus for all $i \geq 0$ a \textbf{restriction map}  
\[ \Res_L: H^i(K,A) \ra H^i(L,A).\]
The map $\Res_L$ is independent of the choice of $\iota$ \cite[$\S$ II.1.1, Remarque 2]{CG}. Whenever $\Res_L$ is defined, we put 
\[ \widetilde{H}^i(L/K,M) = \Ker \left( \Res_L: H^i(K,M) \ra H^i(L,M) \right). \]
 If $\eta \in \widetilde{H}^i(L/K,M)$, we say $L$ is a \textbf{splitting field} for $\eta$. By the definition of Galois cohomology, every $\eta \in H^i(K,M)$ has a finite Galois splitting extension $L/K$. We define the \textbf{index} $I(\eta)$ to be the greatest common divisor of all degrees $[L:K]$ for $L/K$ a finite splitting field of $\eta$. 
For all $\eta \in H^i(K,M)$ with $i \geq 1$ we have $P(\eta) \mid I(\eta)$ and $I(\eta)$ divides some power of $P(\eta)$ \cite[Prop. 11]{WCII}. To explore the relations between period and index is the \textbf{period-index problem in Galois cohomology}.
\\ \\
The following standard result reduces us to the case of prime power period. 

\begin{lemma}(Primary Decomposition)
\label{PDLEMMA}
Let $i,r \geq 1$, and let $\eta_1,\ldots,\eta_r \in H^i(K,M)$ be such that 
$P(\eta_1),\ldots,P(\eta_r)$ are pairwise coprime.   Let $\eta = \eta_1 + \ldots + \eta_r$.  Then 
\begin{equation}
\label{PDEQ1}
 P(\eta) = \prod_{j=1}^r P(\eta_j) 
\end{equation}
and 
\begin{equation}
\label{PDEQ2}
 I(\eta) = \prod_{j=1}^r I(\eta_j). 
\end{equation}
\end{lemma}
\begin{proof}
Easy group theory gives (\ref{PDEQ1}).  For (\ref{PDEQ2}) see e.g. \cite[Prop. 11d)]{WCII}.
\end{proof}

\noindent Here we are interested in the case $M = A(K^{\sep})$ for an abelian variety $A_{/K}$. Then $H^1(K,A)$ is naturally isomorphic to the \textbf{Weil-Ch\^atelet group} of $A_{/K}$, whose elements are torsors under $A_{/K}$.  When $A = E$ is an elliptic curve, classes in $H^1(K,E)$ correspond to genus one curves $C_{/K}$ together with an identification $\Pic^0 C \cong E$.  

\subsection{Constructing WC-Classes With Prescribed Period}

Fix a field $K$ and consider the problem of characterizing all possible periods and indices of elements in $H^1(K,A)$, either for a fixed abelian variety $A_{/K}$ or as $A$ varies in a class of abelian varieties defined over $K$. The point is that it is much easier to understand this problem for the period than for the index. Indeed, in the earliest days of Galois cohomology Shafarevich established the following result.

\begin{thm}(Shafarevich \cite{Shafarevich57})
\label{SHATHM}
Let $K$ be a number field and $A_{/K}$ a nontrivial abelian variety.  For each $n > 1$, there are infinitely many classes $\eta \in H^1(K,A)$ with $P(\eta) = n$.
\end{thm} 
\noindent
We wish to indicate a vast proving ground for the period-index problem in WC-groups by giving a generalization of Theorem \ref{SHATHM}. In order to do so we first define some suitable classes of fields.
\\ \\
A field $K$ is \textbf{MW} if for every abelian variety $A_{/K}$, the group $A(K)$ is finitely generated.  For $n \in \Z^+$, a field $K$ is \textbf{n-WMW} if for every abelian variety $A_{/K}$, the group $A(K)/nA(K)$ is finite.  (The nomenclature may be new, but the study of such fields in the Galois cohomology of abelian varieties goes back to \cite[$\S$ 5]{Lang-Tate58}.)  Finally, $K$ is \textbf{WMW} if it is $n$-WMW for all $n \in \Z^+$.  

\begin{remark}
\label{WMWREMARK}
a) $\F_p$ is a MW field: for any variety $V_{/\F_p}$, $V(\F_p)$ is finite.  \\
b) $\Q$ is a MW field: the \textbf{M}ordell-\textbf{W}eil Theorem.  \\
c) An algebraically closed field $K$ is WMW: for any abelian variety $A_{/K}$, $A(K)$ is a divisible commutative group, so $A(K)/nA(K) = 0$ for all $n \in \Z^+$. \\
d) $\R$ is a WMW field: by Lie theory, for any $g$-dimensional abelian variety $A_{/\R}$, $A(\R)$ is a (split) extension of the finite commutative group $\pi_0(A(\R))$ by the connected component of the identity, 
which is isomorphic to $(\R/\Z)^g$ and thus divisible.  \\
e) $\Q_p$ is a WMW field: by $p$-adic Lie theory (c.f. $\S$ 5.1), for any $g$-dimensional abelian variety $A_{/\Q_p}$, $A(\Q_p)$ is a split extension of a finite commutative group by the group $\Z_p^d$.  \\
f) For any prime number $p$, the Laurent series field $K = \F_p((t))$ is $n$-WMW for all $n$ prime to $p$.  However $K$ is \emph{not} $p$-WMW: see $\S$ 5.1 for stronger results.
\end{remark}

\begin{prop}
\label{LANGNERONPROP}
Let $L/K$ be a finitely generated field extension. \\
a) If $K$ is a MW field, so is $L$.  \\
b) If $K$ is a WMW field, so is $L$.
\end{prop}
\begin{proof}
Let $\{t_1,\ldots,t_n\}$ be a transcendence basis for $L/K$, and put 
$K' = K(t_1,\ldots,t_n)$, so $K'/K$ is finitely generated regular and $L/K'$ is finite.  Thus we may deal separately with the two cases $L/K$ finitely generated regular and $L/K$ finite.   \\
Step 1: Let $L/K$ be finitely generated regular, and let $A_{/L}$ be an abelian variety.  By the Lang-N\'eron Theorem \cite{Lang-Neron59}, \cite[$\S$ 7]{Conrad06} there is an abelian variety $B_{/K}$ and a monomorphism $\tau: B(K) \hookrightarrow A(L)$ with $A(L)/\tau(B(K))$ is finitely generated.  So if $K$ is MW (resp. WMW), so is $L$. \\
Step 2: Suppose $L/K$ is finite. For an abelian variety $A_{/L}$, let $B_{/K}$ be the abelian variety obtained by Weil restriction.  Then $A(L) = B(K)$, hence $A(L)/nA(L) \cong B(K)/nA(K)$ for all $n \in \Z^+$.  Thus if $K$ is MW (resp. WMW), so is $L$.  
\end{proof}
\noindent
Combining Remark \ref{WMWREMARK} and Proposition \ref{LANGNERONPROP} yields: 

\begin{cor}
a) (N\'eron) Every finitely generated field is an MW field.  \\
b) Every field which is finitely generated over an algebraically closed field, $\R$ or $\Q_p$ is WMW. 
\end{cor}
\noindent
We also need the notion of a \textbf{Hilbertian field}.  We refer to the text \cite{Fried-Jarden} both for the definition of Hilbertian field (p. 219) and for the proofs of the following basic results, which are the only facts about Hilbertian fields needed in the proof of Theorem \ref{BIGSHAFTHM} below:
\\ \\
$\bullet$ A Hilbertian field must be infinite.\\
$\bullet$ The field $\Q$ is Hilbertian (Hilbert's Irreducibility Theorem) \cite[Thm. 13.4.2]{Fried-Jarden}.  \\
$\bullet$ If $K$ is Hilbertian and $L/K$ is finite, then $L$ is Hilbertian \cite[Prop. 12.3.3]{Fried-Jarden}.  \\
$\bullet$ For any field $K$, the rational function field $K(t)$ is Hilbertian 
\cite[Thm. 13.4.2]{Fried-Jarden}.
\\ \\
Thus every infinite, finitely generated field (or ``\textbf{IFG field}'') is Hilbertian.  
\\ \\
We can now state the following generalization of Shafarevich's theorem.  

\begin{thm}
\label{BIGSHAFTHM}
 Let $K$ be Hilbertian, and let $A_{/K}$ be a nontrivial abelian variety.  \\
a) If $n > 1$ is an integer indivisible by $\car K$ and such that $A(K)/nA(K)$ is finite, then there are infinitely many classes $\eta \in H^1(K,A)$ of period $n$.  \\
b) If $K$ is WMW, there are infinitely many classes $\eta \in H^1(K,A)$ with any given period $n > 1$.
\end{thm}
\noindent
We will prove Theorem \ref{BIGSHAFTHM} in $\S$ 3.   

\begin{remark}
A field $K$ is \textbf{PAC} (pseudo-algebraically closed) if every geometrically 
integral variety $V_{/K}$ has a $K$-rational point.  Thus all WC-groups 
over a PAC field are trivial.  Algebraically closed fields are PAC and WMW 
but not Hilbertian.  So the hypothesis ``$K$ is Hilbertian'' cannot be omitted from Theorem \ref{BIGSHAFTHM}.  There are fields which are Hilbertian and PAC: it follows from Theorem \ref{BIGSHAFTHM} that for every nontrivial abelian variety $A$ over such a field and all $n \geq 2$, $A(K)/nA(K)$ is infinite.  Thus the hypothesis ``$K$ is WMW'' cannot be omitted from Theorem \ref{BIGSHAFTHM}.
\end{remark}

\subsection{Constructing WC-Classes With Prescribed Index}
 \noindent
It is much harder to construct classes in $H^1(K,A)$ with prescribed index.  This problem was first studied in \cite{Lang-Tate58}.

\begin{thm}(Lang-Tate)
Let $n \in \Z^+$, and let $K$ be a field with infinitely many abelian extensions of exponent $n$.  Let $A_{/K}$ be an abelian variety such that $A(K)/nA(K)$ is finite and $A(K)$ contains at least one element of order $n$. Then $H^1(K,A)$ contains infinitely many elements of period $n$ and, in fact, infinitely many such that the corresponding homogeneous spaces have 
index $n$ as well as period $n$.
\end{thm}
\noindent
It is interesting to compare this result to Theorem \ref{BIGSHAFTHM}: we have the same hypothesis on weak Mordell-Weil groups, and the hypothesis on abelian extensions holds over every Hilbertian field.  On the other hand, the hypothesis on the existence of a point $P \in A(K)$ of order $n$ is prohibitively strong: by Merel's Theorem, for each number field $K$, there are only finitely many $n \in \Z^+$ for which an elliptic curve $E_{/K}$ can have order $n$.  We expect that for any infinite, finitely generated field $K$ and $d \in \Z^+$, torsion is uniformly bounded on $d$-dimensional abelian varieties $A_{/K}$.
\\ \\
Lang and Tate asked whether for every positive integer $n$ there is some elliptic curve $E_{/\Q}$ and class $\eta \in H^1(\Q,E)$ with index $n$. This question remained wide open for years until W. Stein showed \cite{Stein02} that for every number field $K$ and every positive integer $n$ which is not divisible by $8$, there is an elliptic curve $E_{/K}$ and a class $\eta \in H^1(K,E)$ with $I(\eta)= P(\eta) = n$. In \cite{Crelle} the first author showed that for any number field $K$ and every $n \in \Z^+$ there is an elliptic curve $E_{/K}$ and a class $\eta \in H^1(K,E)$ with $I(\eta) = P(\eta) = n$.  
\\ \\
Thus the answer to the question of Lang and Tate is affirmative, but it took almost 50 years to get there.  More recently S. Sharif established a result which we view as a ``complete solution of the period-index problem for elliptic curves over number fields''.  

\begin{thm}(Sharif \cite{Sharif12}) 
\label{BIGSHARIFTHM}
Let $E_{/K}$ be an elliptic curve over a number field, and let $P,I \in \Z^+$ be such that $P \mid I \mid P^2$.  Then there is a class $\eta \in H^1(K,E)$ with $P(\eta) = P$ and $I(\eta) = I$.
\end{thm}
\noindent
We find it plausible that the analogue of Sharif's result should hold for every elliptic curve over every Hilbertian, WMW field $K$.  However, this looks very challenging to prove.  Consider the case of global fields of positive characteristic: i.e., finite field extensions of $\F_p(t)$.  In this case the methods of \cite{Sharif12} work to establish the analogue of Theorem \ref{BIGSHARIFTHM} as long as we require $\gcd(P,p) = 1$.  The case of $p$-power period in characteristic $p$ leads to additional technicalities involving an explicit form of the period-index obstruction map in flat cohomology.  (The first author and S. Sharif have been working on this problem for several years.)  We find it of interest to retreat to the construction of genus one curves with prescribed period over various fields, but where we get to choose the Jacobian elliptic curve to our advantage.  
\\ \\
We can now state the main result of this paper.

\begin{mainthm}
\label{MAINTHM}
Let $K$ be any infinite, finitely generated field.  There is an elliptic curve $E_{/K}$ such that for all $n > 1$, there are infinitely many classes $\eta \in H^1(K,E)$ with $I(\eta) = P(\eta) = n$.
\end{mainthm}

\begin{remark}
The period-index problem has additional subtleties when the period is divisible by the characteristic of the field, e.g. necessitating the use of flat cohomology.  \emph{In Theorem \ref{MAINTHM}, the integer $n$ is allowed to be divisible by the characteristic.}  We feel that a substantial part of the value of Theorem \ref{MAINTHM} lies in its inclusion of this positive characteristic case.
\end{remark}


\begin{remark}
In the statement of the Main Theorem, it is easy to replace the elliptic curve $E_{/K}$ with a $g$-dimensional abelian variety $A_{/K}$ for any fixed $g \geq 1$: for any torsor $C$ under $E$ with $P(C) = I(C) = n$, 
$C \times E^{g-1}$ is a torsor under $E^g$ with $P(C) = I(C) = n$.  It would of course be more interesting to treat abelian varieties and torsors which are not products.
\end{remark}

\begin{remark}
The hypothesis that $K$ is infinite is necessary: by a classic result of Lang \cite{Lang56}, $H^1(\F_q,A) = 0$ for every abelian variety $A$ over a finite field $\F_q$.
\end{remark}

\section{Linear Independence in Torsion Groups}
 \noindent
At several points in the coming proofs, we will find ourselves in the following situation: we have an integer $n > 1$, a commutative group $A$, a subset $S \subset A$ such that every element of $S$ has order $n$, and a homomorphism of groups $R: A \ra B$.  Under what circumstances does $R(S)$ have infinitely many elements of order $n$?  
\\ \\
It suffices for $R$ to be injective, but in our applications this hypothesis is too strong.  On the other hand, if $\Ker R$ is too large then we may have $S \subset \Ker R$.  So it is natural to assume $S$ is infinite and $\Ker S$ is finite.  In this case $R$ has finite fibers and thus $R(S)$ is infinite.  However, it still need not be the case that $R(S)$ has infinitely many elements of order $n$. 

\begin{example}
Let $p$ be a prime and $I$ an infinite set.  Put $A = \left(\bigoplus_{i \in I} \Z/p\Z \right) \oplus \Z/p^2\Z$,  $K = \Z/p\Z^2$, $B = A/K$, and $R: A \ra B$ be the quotient map.  Then \[S = \{1 +x \mid x \in \bigoplus_{i \in I} \Z/p\Z \} \] has cardinality $\# I$ and consists of elements of order $p^2$, but $R(S) \subset B = B[p]$ consists of elements of order dividing $p$.
\end{example}
\noindent
Thus we see that some stronger group-theoretic hypothesis must be imposed, which serves to motivate the definitions of the next paragraph.
\\ \\
Let $n \geq 2$, let $H$ be a commutative group, and let $S = \{\eta_s\} \subset H[n]$.  For $r \in \Z^+$, we say $S$ is \textbf{r-LI} over $\Z/n\Z$ if every $r$-element subset $\{\eta_1,\ldots,\eta_r\}$ is 
linearly independent over $\Z/n\Z$: if for $a_1,\ldots,a_r \in \Z$ we have 
$a_1 \eta_1 + \ldots + a_r \eta_r = 0$, then $a_1 \equiv \ldots \equiv a_r \equiv 0 \pmod{n}$.  Equivalently, the subgroup generated by $\eta_1,\ldots,\eta_r$ has cardinality $r^n$.  We say $S$ is \textbf{LI} over $\Z/n\Z$ if it is $r$-LI over $\Z/n\Z$ for all $r \in \Z^+$.  
\\ \\
In particular, $S$ is $1$-LI over $\Z/n\Z$ iff every element of $S$ has order $n$, and $S$ is 
$2$-LI over $\Z/n\Z$ iff for any distinct elements $\eta_i \neq \eta_j \in S$, we have $\langle \eta_i \rangle \cap \langle \eta_j \rangle = \{0\}$.

\begin{lemma}
\label{LILEMMA}
Let $n \in \Z^+$, let $R: H_K \ra H_L$ be a homomorphism of commutative groups, and let $S \subset H_K[n]$ be a subset.  \\
a) If $\Ker R = 0$, then $\# R(S) = \# S$, and for each $r \in \Z^+$, $S$ is $r$-LI over $\Z/n\Z$ if and only if $R(S)$ is $r$-LI over $\Z/n\Z$.  \\
b) Suppose $\Ker R$ is finite and $S$ is infinite and $2$-LI over $\Z/n\Z$.  Then there is a subset $S' \subset S$ such that $R(S')$ is an infinite $1$-LI over $\Z/n\Z$ subset of $H_L$. 
\end{lemma}
\begin{proof}
Primary decomposition (\ref{PDLEMMA}) reduces us to the case $n = \ell^a$ a prime power.  \\
a) If $\Ker R = 0$, then $R$ is an injection of $n$-torsion groups.  Thus $R: S \ra R(S)$ is a bijection and all linear independence relations are preserved.  \\
b) Let $N = \# \Ker R$.  Let $\eta_1,\ldots,\eta_{N+1}$ be any $N+1$ elements of $S$.  We claim that for at least one $i$, $1 \leq i \leq N+1$, $R(\eta_i)$ has order $n$: if not, then for all $i$ we have 
\[ 0 = \ell^{a-1} R(\eta_i) = R(\ell^{a-1} \eta_i), \]
so $\ell^{a-1} \eta_i \in \Ker R$ for all $i$.  By the Pigeonhole Principle we must have $\ell^{a-1} \eta_i - \ell^{a-1} \eta_j = 0$ for some $i \neq j$, contradicting the hypothesis that $S$ is $2$-LI over $\Z/n\Z$.  This constructs an infinite $S' \subset S$ such that for all $\eta \in S$, $R(\eta)$ has order $n$.   Since $\Ker R$ is finite, all the fibers of $R$ are finite, and thus $R(S')$ is infinite.  
\end{proof}

\section{The Proof of Theorem \ref{BIGSHAFTHM}}

\subsection{The Proof}
\noindent
Step 0: Let $n > 1$ be indivisible by the characteristic of $K$; let $A_{/K}$ be an abelian variety of dimension $g \geq 1$ such that $A(K)/nA(K)$ is finite.  We write $A[n]$ for both the group scheme -- which is \'etale -- and the associated $\gk$-module $A[n](K^{\sep})$ with underlying commutative group $(\Z/n\Z)^{2g}$.  Recall the \textbf{\'etale Kummer sequence}
\[ 0 \ra A(K)/nA(K) \ra H^1(K,A[n]) \ra H^1(K,A)[n] \ra 0. \]
We {\sc claim} that the $n$-torsion group $H^1(K,A[n])$ has an infinite LI over $\Z/n\Z$ subset $S$.  (We only need $S$ to be $2$-LI over $\Z/n\Z$.  But the more 
general result is no harder to prove.)  Applying Lemma \ref{LILEMMA}b) with $H_K = H^1(K,A[n])$, $H_L = H^1(K,A)[n]$ and $R$ the natural map between them, we get an infinite subset $S' \subset S$ whose image in $H^1(K,A)[n]$ is $1$-LI over $\Z/n\Z$.  In other words, $H^1(K,A)$ has infinitely many elements of order $n$, which is what we want to show.  In the remainder of the argument we will establish this claim.  \\
Step 1: Let $L = K(A[n](K^{\sep}))$ -- the fixed field of the kernel of the $\gk$-action on $A[n]$ -- let $\mathfrak{g}_{L/K} = \Aut(L/K)$, and consider the associated inflation-restriction sequence \cite[I.2.6(b)]{CG}
\begin{equation}
\label{INFRES}
 0 \ra H^1(\gg_{L/K},A[n](L)) \ra H^1(K,A[n]) \stackrel{\Phi}{\ra} H^1(L,A[n])^{\mathfrak{g}_{L/K}} 
\stackrel{\partial}{\ra} H^2(\gg_{L/K},A[n](L)). 
\end{equation}
Since $H^i(\gg_{L/K},A[n](L))$ is finite for all $i$, the kernel and cokernel of $\Phi$ are finite.  We will construct an infinite LI over $\Z/n\Z$ subset of $S' =\Phi(H^1(K,A[n]))$.  For each $\eta'_i \in S'$, choose $\eta_i \in \Phi^{-1}(\eta_i')$.  Then $S = \{ \eta_i \}_{i \in S'}$ is an infinite LI over $\Z/n\Z$ subset of $H^1(K,A[n])$.  \\
Step 2:  We have $H^1(L,A[n]) = \Hom(\mathfrak{g}_L,A[n](L))$, so the surjective elements $\eta \in H^1(L,A[n])$  -- i.e., such that $\eta(\mathfrak{g}_L) = A[n](L)$ -- parameterize Galois extensions $M/L$ with $\Gal(M/L) \cong A[n](L) \cong (\Z/n\Z)^{2g}$ together with an isomorphism $\Gal(M/L) \cong A[n](L)$.  The natural action of $\mathfrak{g}_{L/K}$ on this set of order $n$ elements in $H^1(L,A[n])$ consists of precomposing $\chi: \mathfrak{g}_L \ra A[n](L)$ with the automorphism $\sigma^*$ of $\mathfrak{g}_L$ obtained by restricting conjugation by $\sigma \in \mathfrak{g}_K$ to the normal subgroup $\mathfrak{g}_L$.  (Because $A[n](L)$ is commutative, the map $\sigma^*$ depends only on the coset of $\sigma$ modulo $\mathfrak{g}_L$.)  A class $\eta \in H^1(L,A[n])$ is fixed by $\mathfrak{g}_{L/K}$ if and only if the extension $(L^{\sep})^{\Ker \eta}/K$ is Galois.  So we have shown that the surjective elements of $H^1(L,A[n])^{\mathfrak{g}_{L/K}}$ correspond to \textbf{liftings} of the Galois extension $L/K$ to Galois extensions $M/K$.  For $\eta' \in H^1(L,A[n])^{\mathfrak{g}_{L/K}}$, the class $\partial \eta' \in H^2(\mathfrak{g}_{L/K},A[n](L))$ is the class of the extension 
\[ 1 \ra \Gal(M/L) \ra \Gal(M/K) \ra \mathfrak{g}_{L/K} \ra 1, \]
so by exactness of (\ref{INFRES}), we get: $\eta' \in \Phi(H^1(K,A[n]))$ if and only if the above extension splits.  Thus: the surjective elements of $\Phi(H^1(K,A[n]))$ parameterize \textbf{split extensions} of $\mathfrak{g}_{L/K}$ by $A[n](L)$.  \\
Step 3: It was shown by Ikeda \cite{Ikeda60} that over a Hilbertian field, every split embedding problem with abelian kernel $A$ has a proper solution: this means precisely that there is at least one surjective (hence order $n$) element of $\Phi(H^1(K,A[n]))$.  The proof of Ikeda's Theorem (see \cite[$\S$ 16.4]{Fried-Jarden} for a nice modern treatment) goes by specializing a regular $A$-Galois cover of $L(t)$.  Over a Hilbertian field, a regular Galois covering not only has an irreducible specialization but has an infinite linearly disjoint set of irreducible specializations, so we get an infinite set $\{M_i/L\}$ of $A[n](L)$-Galois extensions of $L$ such that $M_i/K$ is Galois and $\Gal(M_i/K) \cong A[n](L) \rtimes \mathfrak{g}_{L/K}$ (see \cite[Lemma 16.4.2]{Fried-Jarden}, which gives a slightly weaker statement; the stronger version needed here follows immediately by an inductive argument).  The linear disjointness of the extensions $M_i/L$ means that the Galois group of the compositum is the direct product of the Galois groups, and this implies that the 
set of classes $\eta_i'$ is LI over $\Z/n\Z$ in $H^1(L,A[n])^{\mathfrak{g}_{L/K}}$, completing the proof.  

\subsection{A Generalization}
\noindent
In the setup of Theorem \ref{BIGSHAFTHM}, the hypothesis that $n$ is not divisible by $\car K$ was used to ensure that $A[n]$ is an \'etale group scheme.  In fact we need only that $A[n]$ admits an \'etale subgroup scheme $H$ of exponent $n$; equivalently, $A(K^{\sep}$) contains a point of order $n$.  For instance this condition holds if $A_{/K}$ is ordinary and $K$ is perfect.  Running the argument with $H$ in place of $A[n]$ gives the following result. 

\begin{thm}
\label{IMPROVEDBIGSHAFTHM}
Let $A_{/K}$ be a nontrivial abelian variety over a Hilbertian field.  Let 
$n > 1$ be such that $A(K)/nA(K)$ is finite.  If $\car K \mid n$, suppose 
$A(K^{\sep})$ contains a point of order $n$.  Then in $H^1(K,A)$ there 
are infinitely many classes of period $n$.
\end{thm}
\noindent

\begin{remark}
For every prime $p$ there is an ordinary elliptic curve $E$ over the perfect field $\F_p$.  Thus for every field $K$ of characteristic $p$ and all $g \geq 1$, there is a $g$-dimensional abelian variety $A$ such that $A(K^{\sep})$ contains points of order $n$ for all $n > 1$.  Thus 
Theorem \ref{IMPROVEDBIGSHAFTHM} implies: every Hilbertian WMW field admits 
abelian varieties whose WC-groups have infinitely many elements of every period $n > 1$.
\end{remark}

\section{The Proof of the Main Theorem: Outline}
\noindent
The proof of the Main Theorem is a three step argument.  We go by induction on the 
transcendence degree, and in fact two out of the three steps go into establishing the base cases.  Here is the first step.  

\begin{thm}
\label{PARTONE}
Let $K$ be a global field, and let $E_{/K}$ be an elliptic curve with 
$E(K) = \Sha(K,E) = 0$.  Then for all $n \geq 1$, there is an infinite 
LI over $\Z/n\Z$ subset $S \subset H^1(K,E)[n]$ such that every $\eta \in S$ has index $n$.  
\end{thm}
\noindent
Theorem \ref{PARTONE} should be compared to \cite[Thm. 3]{Crelle}: the proofs are identical in the number field case, but the logical setup is a bit different: the conclusion of the earlier result is that classes of arbitrary 
index exist in every finite extension $L/K$, but that part of the argument would not work fully in positive characteristic.  In the present proof, controlling the period under extensions $L/K$ is handled in a better way in the inductive step of the argument, which comes later (and which exploits the linear independence property of $S$).  We must give particular care to $p$-primary torsion in characteristic $p$: here this amounts to using Milne's extension of Poitou-Tate global duality.  The proof of Theorem \ref{PARTONE} is given in $\S$ 5.3.
\\ \\
The second step of the proof is to verify that for every \emph{prime global field} $K$ there is an elliptic curve satisfying the hypotheses of Theorem \ref{PARTONE}. In the case of $K = \Q$ we may take the same elliptic curve used in \cite{Crelle}: namely Cremona's $1813B1$ curve
\begin{equation}
\label{CREMONACURVE}
E: y^2+y = x^3-49x - 86. 
\end{equation}
That $E(\Q) = \Sha(\Q,E) = 0$ is a deep theorem of Kolyvagin \cite[Thm. H]{Kolyvagin89}.  
\\ \\
We also we need such a curve $E_{/\F_p(t)}$ for \emph{every prime} $p$.  We will show: 

\begin{thm}
\label{main:IFG_FPt} \label{PARTTWO}
For every prime number $p$, the elliptic curve 
\begin{equation}
\label{ELLIPTICCURVEEQ}
E \, : \, \quad y^2 + txy + t^3y = x^3 + t^2x^2 + t^4x + t^5 
\end{equation}
defined over $\F_p(t)$ has $E (\F_p(t)) = 0$ and $\Sha(\F_p(t),E) = 0$.
\end{thm}
\noindent
The proof of Theorem \ref{PARTTWO} takes advantage of some deep work on the Birch-Swinnerton Dyer conjecture in the function field case.  It is given in $\S$ \ref{Proof:TrivialWM}.
\\ \\
Here is the inductive step.
\begin{thm}
\label{PARTTHREE}
Let $n > 1$, $K$ a WMW field, $A_{/K}$ an abelian variety, and $L/K$ be a finitely generated separable field extension.  Let $S \subset H^1(K,A)[n]$ be an infinite $2$-LI over $\Z/n\Z$ subset.  Then there is an infinite subset $S' \subset S$ such that $\Res_L S' \subset H^1(L,A)[n]$ is infinite and consists of elements of order $n$.  Moreover, if each element of $S$ has index $n$, then each element of $\Res_L S'$ has index $n$.
\end{thm}
\noindent
We will prove Theorem \ref{PARTTHREE} in $\S$ \ref{Proof:InductiveStep}.
\\ \\
We now explain how to put Theorems \ref{PARTONE}, \ref{PARTTWO} and \ref{PARTTHREE} together to prove the Main Theorem.  Let $L$ be an 
infinite, finitely generated field.  Let $k_0$ be its prime subfield: either $\Q$ or $\F_p$.  Since $k_0$ is perfect, $L/k_0$ is separable.  
\\ \\
Case 1: Suppose $k_0 = \Q$.  Then we take the elliptic curve $E_{/\Q}$ 
of (\ref{CREMONACURVE}), with $E(\Q) = \Sha(\Q,E) = 0$.  By Theorem \ref{PARTONE}, for each $n > 1$ there is an infinite LI over $\Z/n\Z$ subset $S \subset H^1(K,E)$ such that every element of $S$ has index $n$.  We apply Theorem \ref{PARTTHREE} with $K = \Q$ and $A = E$ to get the desired result.
\\ \\
Case 2: Suppose $k_0 = \F_p$.  Since $k_0$ is perfect, the finitely generated extension $L/k_0$ admits a separating transcendence basis $t,t_2,\ldots,t_d$.  We take the elliptic curve $E_{/\F_p(t)}$ of Theorem \ref{PARTTWO}.  The rest of the argument proceeds as in Case 1.  

\section{The Proof of Theorem \ref{PARTONE}}

\subsection{Mordell-Weil Groups of Abelian Varieties Over Local Fields}
 \noindent
By a \textbf{local field} we mean a field $K$ which is complete and nondiscrete with respect to an ultrametric norm $|\cdot|$, and with finite residue field $k$. A local field of characteristic $0$ is (canonically) a finite extension of $\Q_p$, and a local field of positive characteristic is (noncanonically) isomorphic to $\F_q((t))$.  
\\ \\
If $K$ is a local field and $G_{/K}$ is an algebraic group, then the set $G$ of $K$-rational points of $G$ has the structure of a \textbf{K-analytic Lie group} in the sense of \cite[LG, Ch. IV]{LALG}.  Since $G$ is quasi-projective, $G$ is homeomorphic to a subspace of $\PP^N(K)$ for some $K$ and is thus second countable.  When $G = A$ is an abelian variety, $G$ is commutative and compact.  In this case, we can use $K$-adic Lie theory to analyze the structure of the Mordell-Weil group and -- crucially for us in what follows -- the weak Mordell-Weil groups $A(K)/nA(K)$.  
\\ \\
This was done somewhat breezily in \cite{Crelle}: we asserted (\ref{LOCALANEQ1}) below, and our justification was ``by $p$-adic Lie theory''.  Here we need also the positive characteristic case, which although certainly known to some experts, to the best of our knowledge does not appear in the literature.  This time around we give a careful treatment of both the $p$-adic and Laurent series field cases here, with an eye towards providing a suitable reference for future work.  


\begin{lemma}
\label{GROUPTOPLEMMA}
Let $H$ be a commutative, torsionfree pro-$p$-group, endowed with its profinite topology.  Then, as a topological group, $H \cong \prod_{i \in I} \Z_p$ for some index set $I$.  If $H$ is second countable, then $I$ is countable.
\end{lemma}
\begin{proof}
The Pontrjagin dual $H^{\vee}$ of $H$ is a commutative $p$-primary torsison group.  Since $H[p] = 0$, $H^{\vee}/H^{\vee}[p] = 0$, so $H$ is divisible.  It is a classical result that a divisible commutative group is a direct sum of copies of $\Q$ and $\Q_p/\Z_p$ for various primes $p$ \cite[5.2.12]{Scott}, and the number of summands of each isomorphism type is invariant of the chosen decomposition.  (Or: an injective module over a commutative Noetherian ring $R$ is a direct sum of copies of injective envelopes of modules of the form $R/\pp$ as $\pp$ ranges over prime ideals of $R$. \cite[Thms. 18.4 and 18.5]{Matsumura}: applying this with $R = \Z_{(p)}$ recovers this classical result.)  Thus for some index set $I$, \[H^{\vee} \cong \bigoplus_{i \in I} \Q_p/\Z_p, \]
and taking Pontrjagin duals gives 
\[ H \cong \prod_{i \in I} \Z_p. \]
An uncountable product of second countable Hausdorff spaces, each with more than a single point, is not second countable \cite[Thm. 16.2c)]{Willard}, so if $H$ is second countable, $I$ is countable.
\end{proof}

\begin{thm}
\label{LOCALANALYTICTHM}
Let $K$ be a local field, with valuation ring $R$, maximal ideal $\mm$ and residue field $\F_q = \F_{p^a}$.  If $\car K = 0$, let $d = [K:\Q_p]$.  Let $G$ be a compact commutative second countable $K$-analytic Lie group, of dimension $g \geq 1$.  \\
a) If $\car K = 0$, then $G[\tors]$ is finite and we have a topological group isomorphism
\begin{equation}
\label{LOCALANEQ1}
G \cong \Z_p^{dg} \oplus G[\tors]. 
\end{equation}
b) If $\car K = p$, then $G[\tors]$ is finite if and only if $G[p]$ is finite.  When these 
equivalent conditions hold -- e.g. when $G = A(K)$ for a semi-abelian variety 
$A_{/K}$ -- then we have a group isomorphism
\begin{equation}
\label{LOCALANEQ2}
G \cong \left( \prod_{i=1}^{\infty} \Z_p \right) \oplus G[\tors]. 
\end{equation}
\end{thm}
\begin{proof}
Step 0: By \cite[LG IV.9]{LALG}, $G$ has a filtration by open subgroups 
\[ G = G^{-1} \supset G^0 \supset G^1 \supset \ldots \supset G^n \supset \ldots\]
such that: \\
(i) Each $G^i$ is obtained by evaluating a $g$-dimensional formal group law on $(\mm^i)^g$ ; \\
(ii) $\bigcap_{i \geq 1} G^i = \{0\}$;  \\
(iii) For all $i \geq 1$, $G^i/G^{i+1} \cong (k,+)$ is finite of exponent $p$; and \\
(iv) $G^1[\tors] = G^1[p^{\infty}]$ \cite[LG 4.25, Thm. 3]{LALG}.
\\ \\
Step 1: We show that $G[\tors]$ is finite if and only if $G[p]$ is finite.  (They need \emph{not} hold in characteristic $p$; take $G = \G_a = (R,+)$.)  Clearly $G[\tors]$ finite implies $G[p]$ finite.  Conversely, if $G[p]$ 
is finite, then it follows from (iv) and (ii) that $G^n$ is torsionfree for all sufficiently large $n$, and then (iii) implies that $G[\tors]$ is finite.  \\
Step 2: From now on we assume $G[\tors]$ is finite.  Thus $G[\tors]$ has finite exponent, and then \cite[Thm. 8;5]{Baer36} implies that there is a torsionfree subgroup $H$ of $G$ such that
\[ G = H \oplus G[\tors]. \]  
Thus (at least as an abstract group), $H$ is isomorphic to the torsionfree profinite commutative group $G/G[\tors]$.  Since $G$ has a finite index pro-$p$-subgroup (namely $G^n$ for sufficiently large $n$, it follows that $H$ is pro-$p$, so by Lemma \ref{GROUPTOPLEMMA}, \[H \cong G/G[\tors] \cong \prod_{i \in I} \Z_p \] for some index set $I$, and since $G/G[\tors]$ is second countable, $I$ is countable.  Thus 
\[ G \cong \left( \bigoplus_{i \in I} \Z_p \right) \oplus G[\tors]. \]
Step 3: Suppose $\car K  = 0$.  Then $K$-adic Lie groups $G$ and $H$ have isomorphic open subgroups if and only if their Lie algebras are isomorphic \cite[LG 5.34 Cor. 1]{LALG}.  So every $g$-dimensional compact, commutative $K$-adic Lie has an open subgroup isomorphic to $R^g \cong \Z_p^{dg}$.  This applies to our $G$ and gives in particular that $G$ is topologically finitely generated, so all finite index subgroups of $G$ are open.   Thus $H \cong \prod_{i \in I} \Z_p$ and $\Z_p^{dg}$ are both open subgroups of $G$; it follows easily that $H \cong \Z_p^{dg}$.  This completes the proof of part a).  \\
Step 4:  If $\car K = p > 0$, then in the formal group $G_1$, we have $[p] \in R[[X_1^p,\ldots,X_g^p]]^g$ \cite[LG 4.21 Cor. and LG 4.29 Exc. 7]{LALG}.  (A different proof using invariant differentials is given in the $g = 1$ case -- which is the case of our Main Theorem -- in \cite[Cor. 4.4]{AECI}.  It is straightforward to adapt this argument to the general case using the corresponding properties of invariant differentials on abelian varieties.)  Consider $pG_1$ as a subset of the profinite space $G_1 = (t\F_q[[t]])^g$.  It is compact, hence closed.  Moreover, it lies in $(t\F_q[[t^p]])^g$, so it is not open.  We deduce that $G_1/pG_1$ is infinite.  Since 
$G_1$ and $H$ are finite index subgroups of $G$, it follows that $H/pH$ is infinite, and thus $I$ is infinite.  Since $I$ is countable, 
\[ G \cong \prod_{i=1}^{\infty} \Z_p \oplus G[\tors]. \]
\end{proof}
\noindent
We immediately deduce the following result.

\begin{cor}
\label{LOCALANALYTICCOR}
We retain the notation of Theorem \ref{LOCALANALYTICTHM}.  \\
a) If $A(K)$ contains a point of order $n$, then so does $A(K)/nA(K)$.  \\
b) For all $a \geq 1$, $A(K)/p^a A(K)$ contains a point of order $p^a$.  \\
c) If $\car K > 0$, then for all $a \geq 1$, $A(K)/p^a A(K)$ contains 
an infinite subsset which is LI over $\Z/p^a\Z$.   
\end{cor}
\noindent
The following key technical result records a global consequence of this local analysis.  

\begin{lemma}
\label{KEYDUALITYLEMMA}
Let $K$ be a global field, and let $A_{/K}$ be an abelian variety of dimension $g \geq 1$.  Let $n > 1$ be an integer.  \\
a) If $\car K \nmid n$, then there is a positive density set $\mathcal{P}$ of finite places of $K$ such that for all $v \in \mathcal{P}$, $H^1(K_v,A)$ has an element of order $n$.\\
b) If $\car K = p$ is a prime and $n = p^a$ for $a \geq 1$, then for every place $v$ of $K$, $H^1(K_v,A)$ has infinitely many elements of order $n$.
\end{lemma}
\begin{proof}
Step 1: For any finite place $v$ of $K$, the discrete torsion group $H^1(K_v,A)$ is Pontrjagin dual to the compact profinite group $A(K_v)$: this celebrated \textbf{Local Duality Theorem} is due to Tate \cite[Proposition 1]{Tate58} when $\car K = 0$ and to Milne \cite{Milne70} when $\car K > 0$.  It follows that for $n \geq 1$, the groups $H^1(K_v,A)[n]$ and $A(K_v)/nA(K_v)$ are Pontrjagin dual.
\\
Step 2: Suppose $\car K \nmid n$.  Then, as already recalled, $A[n]$ is a finite \emph{\'etale} group scheme, so there is a finite Galois extension $L/K$ such that $A(L)[n] \cong (\Z/n\Z)^{2g}$.  By the Cebotarev Density Theorem, the set of finite places $v$ which split completely in $L$ has positive density (which one can explicitly bound below in terms of $n$ and $g$, if needed).  For each such $v$, there is a $K$-algebra embedding $L \hookrightarrow K_v$ and thus $A(K_v)[n] \cong (\Z/n\Z)^{2g}$.  By Corollary \ref{LOCALANALYTICCOR}, $A(K_v)/nA(K_v)$ has a point of order $n$ (in fact $2g$ points which are LI over $\Z/n\Z$), so by Local Duality so does $H^1(K_v,A)$.  This establishes part a).  \\
Step 3: Suppose $\car K = p$ and $n = p^a$ for $a \geq 1$.  In this case $A[p^a]$ is never \'etale and need not admit an \'etale subgroup scheme of exponent $p^a$: c.f. Remark \ref{PTORSIONREMARK}, so the argument of Step 2 breaks down.  Fortunately it is not needed.  Combining Corollary \ref{LOCALANALYTICCOR} with Local Duality yields the (stronger!) result in this case.
\end{proof}

\subsection{A Local-Global Isomorphism in WC-Groups}
\noindent
Let $K$ be a global field, $A_{/K}$ an abelian variety, and $p$ 
be a prime number.  Put 
\[ T_p \Sel A = \varprojlim_n \Ker \left( H^1(K,A[p^n]) \ra \bigoplus_{v \in \Sigma_K} 
H^1(K_v,A[p^n]) \right) . \]
We need the following result of Gonz\'alez-Avil\'es-Tan, a generalization of work of Cassels-Tate.  In what follows, $A^{\vee}$ denotes the dual abelian variety of the abelian variety $A$, $G^{\wedge}$ denotes the pro-$p$-completion of the commutative group $G$, and $G^*$ denotes the Pontrjagin dual of the commutative group $G$.  

\begin{thm}
For $A_{/K}$ an abelian variety over a global field, and $p$ any prime number -- the case $p = \car K$ is allowed -- we have an exact sequence
\begin{equation}
\label{GATEQ}
 0 \ra T_p \operatorname{Sel} A^{\vee} \ra \prod_{v \in \Sigma_K} (A^{\vee}(K_v))^{\wedge} \stackrel{\alpha}{\ra} (H^1(K,A)[p^{\infty}])^* \ra (\Sha(K,A)[p^{\infty}])^* \ra 0.
\end{equation}
Using Tate-Milne Local Duality to identify $H^1(K_v,A)$ and $A^{\vee}(K_v)$ as Pontrjagin duals, the map $\alpha$ is the Pontrjagin dual of the natural map \[H^1(K,A)[p^{\infty}] \ra \bigoplus_{v \in \Sigma_K} H^1(K_v,A)[p^{\infty}]. \] 

\end{thm}
\begin{proof}
This is the main result of \cite{GAT07}.
\end{proof}

\begin{cor}
\label{GATCOR}
Let $A_{/K}$ be an abelian variety defined over a global field.  If $A^{\vee}(K) = \Sha(K,A) = 0$, then the local restriction maps induce an isomorphism of groups
\[  H^1(K,A) \stackrel{\sim}{\ra} \bigoplus_{v \in \Sigma_K} H^1(K_v,A). \]
\end{cor}
\begin{proof}
Since WC-groups are torsion, it is enough to restrict to $p$-primary components for all primes $p$.  Since $A^{\vee}(K) = 0$, $T_p \Sel A^{\vee} = 0$, and then (\ref{GATEQ}) gives an isomorphism
\[ \prod_{v \in \Sigma_K} (A^{\vee}(K_v))^{\wedge} \stackrel{\sim}{\ra} H^1(K,A)[p^{\infty}]. \]
Taking Pontrjagin duals and applying Tate-Milne Local Duality, we get 
\[   H^1(K,A)[p^{\infty}] \stackrel{\sim}{\ra} \bigoplus_{v \in \Sigma_K} H^1(K_v,A)[p^{\infty}].  \]
\end{proof}

\begin{lemma}
\label{5.3}
Let $E_{/K}$ be an elliptic curve over a global field, and let $\eta \in H^1(K,E)$ be locally trivial at all places of $\Sigma_K$ except (possibly) one.  Then $P(\eta) = I(\eta)$.
\end{lemma}
\begin{proof}
In the number field case this is \cite[Prop. 6]{Crelle}.  Two proofs are given.  The second proof works verbatim in the function field case.  The first proof, which makes use of the period-index obstruction map $\Delta$, works if one uses the extension of $\Delta$ given in \cite[$\S$ 2.3]{WCIV}.
\end{proof}

\subsection{Proof of Theorem \ref{PARTONE}}
 \noindent
Let $K$ be a global field, and let $E_{/K}$ be an elliptic curve with $E(K) = \Sha(K,E) = 0$.  Let $n > 1$; we must show that there is an infinite LI over $\Z/n\Z$ subset of $H^1(K,E)$ consisting of classes with index $n$.
\\ \\
By Corollary \ref{GATCOR} have an isomorphism 
\begin{equation}
\label{PARTONEEQ1}
H^1(K,E) \stackrel{\sim}{\ra} \bigoplus_{v \in \Sigma_K} H^1(K_v,E). 
\end{equation}
With all our preparations in hand, the proof is simple: for each of an infinite set $\mathcal{P}$ of finite places $v$ of $K$, we find a class $\eta_v \in H^1(K_v,E)$ of period $n$, realize this class as an element of the right hand side of (\ref{PARTONEEQ1}) supported at $v$, and pull back via the isomorphism to get a class $\eta$ in $H^1(K,E)$. This class has period $n$, and since it is locally trivial except at $v$, by Lemma \ref{5.3} it also has index $n$. Doing this for each $v \in \mathcal{S}$ we get a infinite subset $S \subset H^1(K,E)[n]$ which is LI over $\Z/n\Z$ because each class lies in a different direct summand. We implement this in several steps, corresponding to the preliminary results we have established.  
\\ \\
Step 1: Suppose $\car K \nmid n$.  We apply Lemma \ref{KEYDUALITYLEMMA}a) to get our infinite set $\mathcal{P}$ of finite places of $v$ and $\eta_v \in H^1(K_v,E)$ of order $n$.  This completes the proof if $\car K = 0$.  \\
Step 2: Suppose $\car K = p > 0$ and $n = p^a$.  We apply Lemma \ref{KEYDUALITYLEMMA}b) and may take $\mathcal{P} = \Sigma_K$. \\
Step 3: Finally, suppose $\car K = p > 0$ and write $n = p^a m$ with $\gcd(m,p) = 1$.  By Step 1, there is an infinite LI over $\Z/m\Z$ subset $S_1 \subset H^1(K,E)[m]$ such that every $\eta_1 \in S_1$ has index $m$.  By Step 2, there is an infinite LI over $\Z/p^a\Z$ subset $S_2 \subset H^1(K,E)[p^a]$ such that every $\eta_2 \in S_2$ has index $p^a$.  Using Lemma \ref{PDLEMMA} we find (easily) that $S_1 + S_2$ is an infinite LI over $\Z/n\Z$ subset of $H^1(K,E)[n]$ such that every $\eta \in S$ has index $n$.

\section{The Proof of Theorem \ref{PARTTWO}}
\label{Proof:TrivialWM}
\noindent
Now we will prove Theorem \ref{PARTTWO}: for every prime number $p$, the elliptic curve 
\[ E_{/\F_p(t)}: y^2 + txy + t^3 y = x^3+ t^2 x^2 + t^4x + t^5 \]
has trivial Mordell-Weil and Shafarevich-Tate groups.  

\subsection{Controlling the torsion}
\label{NoTorsion}
 \noindent
We deal with the $p$-primary torsion and prime-to-$p$ torsion in $E(\F_p(t))$ separately. 

\begin{lemma}
\label{LemmaNoPTorsion}
Let $k$ be a field of characteristic $p>0$, and $E_{/k}$ an elliptic curve.  If $E(k)[p^{\infty}] \neq 0$, then $j(E) \in k^p$.
\end{lemma}

\begin{proof}
Let $P \in E(k)$ be a point of order $p$. Let $E' = E / \langle P \rangle$ be the quotient of $E$ by the cyclic group generated by $P$.  We have a separable isogeny $\Phi : E \to E'$ with kernel $\langle P \rangle$ and of degree $p$.  If $\Phi^{\vee} : E' \to E$ is its dual isogeny, we have a factorization of multiplication by $p$ on $E$ as 
\[ [p] : E \stackrel{\Phi}{\lra} E' \stackrel{\Phi^{\vee}}{\lra} E.\]
Since $[p] :E \to E$ is inseparable of degree $p^2$, we must have that $\Phi^{\vee}$ is inseparable of degree $p$.  But an elliptic curve in characteristic $p$ has a unique inseparable isogeny of degree $p$, namely the quotient by the kernel of Frobenius, so $\Phi^{\vee}$ must be the Frobenius map on $E'$, and thus $E \cong (E')^{(p)}$ and $j(E) = j((E')^{(p)}) = (j(E'))^p \in k^p$.
\end{proof}
\noindent
We get as an immediate consequence:
\begin{cor}
\label{PTORSIONPCOR}
Let $E_{/\Fp(t)}$ an elliptic curve.  If $j(E) \notin \Fp(t^p)$, then $E(\overline{\Fp}(t))[p^{\infty}] = 0$.
\end{cor}
\begin{remark}
\label{PTORSIONREMARK}
In the setting of Corollary \ref{PTORSIONPCOR} we have that $E$ is an ordinary elliptic curve which has no $p$-torsion over the separable closure of $\F_p(t)$.  This is a concrete instance of a phenomenon encountered in $\S$ 3.2.  In particular, it serves to clarify why assuming ``$A$ is ordinary'' in Theorem \ref{BIGSHAFTHM} would not be enough (in order for our argument to succeed, at least).
\end{remark}
\noindent
To control the prime-to-$p$ torsion we use the following standard strategy: let $E_{/K}$ be an elliptic curve over a global field, let $v$ be a finite place of $K$, and denote by $K_{v}$ the corresponding completion.  Since $E(K) \hra E(K_{v})$, it suffices to find a place for which $E(K_{v})$ contains no prime-to-$p$ torsion.  For a group $G$, we denote by $G[p']$ its prime-to-$p$ torsion subgroup.  Much as above, we let $R_{v}$ the valuation ring of  $K_{v}$, $\f{m}_{v}$ the maximal ideal of  $R_{v}$, and $k_{v}$ its residue field, of characteristic $p$.  We denote by $\tilde{E}_{/k_{v}}$ the reduction of $E$ modulo $\f{m}_{v}$, $\tilde{E}_{\ns}(k_{v})$ the set of nonsingular points of $\tilde{E}(k_{v})$, $E_{0}(K_{v})$ the set of points of $E(K_{v})$ with nonsingular reduction, and $E_1(K_{v})$ the kernel of the reduction map $E(K_{v}) \to \tilde{E}(k_{v})$.  We have a short exact sequence
\[0 \lra E_0(K_{v})  \lra   E(K_{v})  \lra  E(K_{v}) / E_0(K_{v}) \lra  0\]
The group $E(K_{v}) / E_0(K_{v}) $ is always finite, and its order is the number of reduced geometric components of the special fiber of a minimal regular model of $E$ over $R_{v}$ (see e.g. \cite[Corollary IV.9.2(d)]{AECII}).  We see from \cite[Table 4.1]{AECII} that $E(K_{v}) / E_0(K_{v})$ is the trivial group whenever the special fiber is of type II or II${}^*$, and in those cases we have $\tilde{E}_{\ns}(k_{v}) = k^{+}_{v}$.  In this case, by \cite[Proposition VII.2.1]{AECI} we have a short exact sequence 
\[0 \lra E_1(K_{v}) \lra E(K_{v}) \lra k^{+}_{v} \lra 0\]
As recalled in $\S$ 4.2.2, $E_1(K_{v})$ is obtained from a formal group law, so contains no prime-to-$p$ torsion.  In particular, if $j(E) \notin (K_{v})^p$, Lemma \ref{LemmaNoPTorsion} and the above short exact sequence imply 
\[ E_1(K_{v})[\tors] = E_1(K_{v})[p^{\infty}] \subset E(K_{v})[p^{\infty}] =0.\]
In this case we have an injection
\[E(K_{v}) [\tors] = E(K_{v})[p'] \hra k^{+}_{v}.\]
Since $k^{+}_{v}$ is a $p$-group, we conclude that $E(K_{v})[p'] = 0$.
\\ \\
We will see that for the elliptic curve in Theorem \ref{PARTTWO}, there is always a place $\nu$ of $\F_p(t)$ for which the fiber of a minimal model for $E$ at $\nu$ is of type II${}^{*}$.

\subsection{Controlling the rank}
\label{Control:rank}
 \noindent
Let $k$ be a field, $C$ a smooth projective curve over $k$, and $K = k(C)$ the function field of $C$.  Let $E_{/K}$ be an elliptic curve, and consider $\pi : \c{S} \to C$ its associated minimal elliptic surface. We will always assume that $E$ is \textbf{nonisotrivial}, i.e., $j(E) \notin k$.  By Lang-N\'eron, this implies that $E(K)$ is finitely generated.  It implies also that $\Delta (E) \notin k$, so the morphism $\pi : \c{S} \to C$ contains singular fibers, and that the \textbf{N\'eron-Severi group} $\NS (\c{S})$ is finitely genereated.   
The ranks of $E(K)$ and $\NS (\c{S})$ are related by the following result

\begin{thm}[Shioda-Tate]
Let $k$ be an algebraically closed field, let $E_{/ k(C)}$ be an nonisotrivial elliptic curve, and let $\pi: \c{S} \to C$ be its associated minimal elliptic surface.  Let $\Sigma$ denote the finite set of points $v \in C$ for which the fiber $\pi^{-1}(v)$ is singular.  For each $v \in \Sigma$, let $m_{v}$ denote the number of irreducible components of $\pi^{-1}(v)$.  We have
\[\rank( \NS (\c{S}) ) =  \rank (E)  + 2 + \sum_{v \in \Sigma}{(m_{v} - 1)}.\]
\end{thm}
\begin{proof}
See  \cite[Cor. 1.5]{Shioda72}.
\end{proof}
\noindent
Now let $C = \PP^1$, so $k(C) \cong k(t)$.  Then $\c{S} \to \PP^1$ admits a Weierstrass equation
\[\c{S} = \{([X:Y:Z],t) \in \PP^2 \times \PP^1 : Y^2Z + a_1XYZ + a_3YZ^2 = X^3 + a_2X^2Z + a_4XZ^2 +a_6Z^3\} \]
for some $a_i(t) \in k[t]$.  We define the \textbf{height} of the Weierstrass elliptic surface to be the least $n \in \N$ such that $\deg (a_i) \leq ni$ for all $i$. The height controls the geometry of the total space $\c{S}$.  If $E_{/k(t)}$ has height $n=1$, the associated minimal elliptic surface $\c{S}$ is isomorphic to $\PP^2$ blown up at $9$ points, so the rank of $\NS (\c{S})$ is  $10$. See \cite[Lemma 10.1]{Shioda90} for a different computation of this. Therefore, the Shioda-Tate formula allows us to compute the rank of a height 1 elliptic curve $E_{/k(t)}$ from the local information of the singular fibers:
\begin{cor}
\label{ShiodaRank}
For a nonisotrivial elliptic curve $E_{/ k(t)}$ of height 1, we have
\[ \rank E = 8 - \sum_{v \in \Sigma}{(m_{v} - 1)},\]
where $\Sigma$ denotes the places of $k(t)$ where $E$ has bad reduction.
\end{cor}
\noindent
Our strategy is to find an elliptic curve for which the contribution from the singular fibers is exactly $8$, so $\rank E_{\overline{k}(t)} = 0$ and \emph{a fortiori} $\rank E_{k(t)} = 0$.  By Shioda-Tate, this occurs if there is a place $v$ of bad reduction with $m_v = 9$.  So we choose the Weiertrass equation in Theorem \ref{PARTTWO} so as to ``force'' Tate's algorithm to give us one fiber of reduction type II$^*$ (so $m_v = 9$).

\subsection{Controlling $\Sha$}
\label{ControllingSha}
 \noindent
In order to compute (the size of) the Shafarevich-Tate group of the elliptic curve in Theorem 1 we use some special features of elliptic curves of small height over function fields.
\begin{thm}
\label{BSD}
Let $E_{/\F_q(t)}$ be an elliptic curve of height $\leq 2$.  Then $E$ satisfies the Birch and Swinnerton-Dyer conjecture:
\begin{equation}
\label{BSDformula}
\lim_{s \to 1} \frac{L(E,s)}{(s-1)^r} = \frac{\# \Sha (\F_q(t),E) \cdot R \cdot \tau}{(\# E(\F_q(t))[\tors])^2},
\end{equation}
where $r = \rank (E)$, $\tau$ is the product of the Tamagawa numbers, $R$ is the regulator of $E$, and $L(E,s)$ is the $L$-function of $E$.
\end{thm}
\begin{proof}
Write $K = \F_q(t)$.  The main result of \cite{KatoTrihan} (specialized to elliptic curves) states that (\ref{BSDformula}) holds whenever $\Sha(K,E)$ is finite (in fact, it suffices that $\Sha(K,E)[l]$ is finite for some prime $l$).  The finiteness of $\Sha (K,E)$ is known for elliptic curves over $K$ of height $\leq 2$.  We refer the reader to \cite[Lectures 2,3]{Ulmer11} and the references therein for details.
\end{proof}
\noindent
Therefore, we can compute the order of $\Sha(\F_p(t),E)$ from the $L$-function of $E$, $E(\F_p(t))$ and data from Tate's algorithm. As we will see, the $L$-function of the elliptic curve (\ref{ELLIPTICCURVEEQ})  contributes trivially, by virtue of the following result.
\begin{prop}[Grothendiek, Raynaud, Deligne]
\label{degreeLF}
Let $E_{/\F_q(t)}$ be a nonisotrivial elliptic curve. Then the L-function of $E$, $L(E,s)$, is a polynomial in $\Z [q^{-s}]$ with constant coefficient $1$ and degree $\deg (\f{n}) - 4$, where $\f{n} = \f{n}(E)$ is the conductor of $E$.
\end{prop}
\begin{proof}
See e.g. \cite[Theorem 2.6]{Gross11}.
\end{proof}

\subsection{\label{PPT}Proof of Theorem \ref{PARTTWO}}
 \noindent
Now we will put together the pieces to show that the elliptic curve 
\[ E_{/\F_p(t)} \, : \, \quad y^2 + txy + t^3y = x^3 + t^2x^2 + t^4x + t^5 \]
of Theorem \ref{PARTTWO} has $E(\F_p(t)) = \Sha(\F_p(t),E) = 0$.
\\ \\
First, the discriminant and $j$-invariant of $E$ are given by
\[ \Delta (E) = -t^{10}(83t^2 -199t +432), \] 
\[ j(E) = - \frac{(47)^3t^{12}}{\Delta} =   \frac{(47)^3t^2}{83t^2 -199t +432}\]
so we do indeed have an elliptic curve for all $p$.  For $p \neq 47$, $j(E) \notin (\F_p(t))^p$, so by Lemma \ref{LemmaNoPTorsion}, $E(\F_p(t))$ has no $p$-primary torsion.  We verify, using Tate's algorithm that $E$ has reduction of type II${}^*$ at $(t)$, so in particular it has additive reduction with trivial component group, so by the results of \S \ref{NoTorsion}, we conclude that $E(\F_p(t))[\tors] = 0$.  For $p=47$, $j(E) = 0$, so $E$ is isotrivial and $j(E)$ is a $p$th power -- so we verify using Magma that $E(\F_{47}(t)) = 0$.
\\ \\
To compute the rank of $E(\F_p(t))$ we examine the other singular fibers: for $p = 2$  and $p=3$, we have $\Delta (E) = t^{11}(t+1)$, so $E$ has bad reduction also at $(t+1)$.  For $p > 3$ and $p\neq 83$, $E$ has one or two more places of bad reduction, depending whether the quadratic $83t^2 -199t +432$ factors over $\F_p[t]$ into two (different) linear factors or not.  Finally, for $p = 83$, $E$ has bad reduction at $(t+2)$ and at the place at infinity $(1/t)$.  In any case we verify, using Tate's algorithm, that $E$ has reduction type II$^{*}$ at $(t)$ and reduction type I$1$ and the other place(s), so Corollary \ref{ShiodaRank} gives
\[ \rank(E) =  8 - (9-1) = 0.\]
Thus $E(\F_p(t)) = 0$ for all primes $p$.
\\ \\
Now we compute the order of $\Sha (\F_p(t),E)$: $E$ has height $1$, and since $E(\F_p(t)) = 0$ we have $r = 0$, $\# E(\F_p(t))[\tors] =1$ and $R=1$.  Fibers of type II${}^*$ and I$1$ have both Tamagawa number $1$, so $\tau =1$.  Thus formula (\ref{BSDformula}) reduces to $L(E,1) = \# \Sha(E)$, so we need to compute the $L$-function of $E$.  Luckily for us, $E$ has trivial $L$-function:  The support of the conductor is given by the places of bad reduction of $E$, discussed in the previous paragraph, and we use Ogg-Saito formula to compute the exponent of the conductor at these places: 
\[ \f{n}(E) =
\left\{
\begin{array}{cl}
 3(t) + (t+1) &, \, p =2,3 \\
2(t) + (\mbox{Linear}_1) + (\mbox{Linear}_2) \mbox{  or  } 2 (t) + (\mbox{Quadratic}) &, \, p>3 , p \neq 83 \\
2(t) + (t+2) + (1/t) &, \,  p = 83
\end{array}
\right.\]

In any case, $\deg (\f{n}) = 4$, so by Proposition \ref{degreeLF} we have $L(E,s) = 1$. Thus $\# \Sha ( \F_p(t),E) =1$ for all primes.  This completes the proof of Theorem \ref{PARTTWO}.
\begin{remark}
When $p \neq 47$, our arguments show $E(\overline{\Fp}(t)) = 0$. However the case $p = 47$ is really exceptional: here $j(E) = 0$ so $E$ is isotrivial. Since $47 \equiv -1 \pmod{3}$, by Deuring's Criterion $E$ is supersingular, so $E(\overline{\F_p(t)})[p] = 0$.  We still have a fiber of type II$^*$, so $E(\overline{\Fp}(t))[\tors] = 0$.  
\end{remark}

\section{The Proof of Theorem \ref{PARTTHREE}}
\label{Proof:InductiveStep}
\subsection{Some preliminaries}
\noindent
Let $A_{/K}$ be an abelian variety.  We want to give conditions on a field extension $L/K$ for the group $\widetilde{H}^1(L/K,A)$ to be finite.  We treat regular extensions first, then finite extensions.  

\begin{lemma}
\label{THREELEMMA1}
Let $L/K$ be a purely transcendental field extension.  \\
a) Let $V_{/K}$ be an algebraic variety.  Suppose \emph{either} that $K$ is infinite or $V$ is complete. Then $V(L) \neq \varnothing \implies V(K) \ \neq \varnothing$.  \\
b) For every abelian variety $A_{/K}$, we have $\widetilde{H}^1(L/K,A) = 0$.  
\end{lemma} 
\begin{proof} a) 
Step 1: Let $\{t_i\}_{i \in I}$ be a transcendence basis for $L/K$.  If $P \in V(L)$, there is a finite subset $J \subset I$ such that $P \in V( K(\{t_i\}_{i \in J})$.  So we may assume that $L/K$ has finite transcendence degree.  Induction reduces us to the case $L = K(t)$.  \\
Step 2: A point $P \in V(K(t))$ corresponds to a rational map $\varphi: \mathbb{P}^1 \ra V$. The locus on which $\varphi$ is not defined is a finite set of closed points of $\mathbb{P}^1$.  If $K$ is infinite, so is $\mathbb{P}^1(K)$, so there is $P \in \mathbb{P}^1(K)$ at which $\varphi$ is defined, and then $\varphi(P) \in V(K)$. Moreover any rational map from a regular curve to a complete variety is a morphism, so if $V$ is complete then e.g. $\varphi(0) \in V(K)$.  \\
b) Since $\eta \in H^1(K,A)$ corresponds to a torsor $V$ under $A$ and thus a projective variety, this follows immediately from part a).  
\end{proof}

\begin{remark}
a) If in the statement of Lemma \ref{THREELEMMA1}a) we strengthen ``complete'' to ``projective'', there is an easier proof: let $\varphi: V \ra \mathbb{P}^N$ be a $K$-embedding.  Since $K(t)$ is the fraction field of the UFD $K[t]$, if $P \in V(L)$, we can write $\varphi(P) = [f_0(t):\ldots:f_N(t)]$ with $\operatorname{gcd}(f_0,\ldots,f_N) = 1$.  In particular, some $f_i(t)$ is not divisible by $t$ and thus $(f_0(0):\ldots:f_N(0)) \in V(K)$.  \\
b) Let $K = \F_q$ be a finite field.  Then the affine curve $V = \PP^1_{\F_q} \setminus \PP^1(\F_q)$ has $K(t)$-rational points but no $K$-rational points.
\end{remark}
\noindent
Let $K$ be a field, $M$ a commutative $\gk$-module, $i \geq 1$, $L/K$ a field extension, and consider the restriction map $\Res_L: H^i(K,M) \ra H^i(L,M)$. For $\eta \in H^i(K,M)$ it is natural to compare both the period and index of $\eta$ to the period and index of $\Res_L \eta$. Let us say that the extension $L/K$ is \textbf{index-nonreducing} if $I(\Res_L \eta) = I(\eta)$ for all $\eta$ and \textbf{period-nonreducing} if $P(\Res_L \eta) = P(\eta)$ for all $\eta$. It is then easy to see:
\\ \\
$\bullet$ If $L/K$ is index-nonreducing, it is also period-nonreducing.\\
$\bullet$ $L/K$ is period-nonreducing if and only if $\widetilde{H}^i(L/K,M) = 0$.\\ Thus Lemma \ref{THREELEMMA1} may be viewed as a result on period-nonreduction.
\\ \\
However an extension can be period-nonreducing but not index-nonreducing.  In our context the difference is immaterial, because our inductive argument gives us 
classes with period equals index, but in general it would be more useful to have index-nonreduction results.  So in the interest of completeness and applicability to future work we also include the following result.  

\begin{prop}
\label{GEOPROP}
Let $X_{/K}$ and $V_{/K}$ be regular, geometrically integral varieties with $X_{/K}$ complete.\footnote{It is enough to assume that $X$ admits a resolution of singularities.}   Then: \\
a) We have $I(V) \mid I(X) I(V_{/K(X)})$.  \\
b) If $I(X) = 1$ -- in particular if $X(K) \neq \varnothing$ -- we have $I(V_{/K(X)}) = I(V)$.  
\end{prop}
\begin{proof}
a) It suffices to show the following: for every finite splitting extensions $M$ of $V_{/K(X)}$, there is a $K$-rational zero-cycle on $V$ of degree $[M:K(X)]I(X)$. Let $L$ be the algebraic closure of $K$ in $M$, so there is an $L$-variety $\tilde{X}$ and a dominant morphism $\pi: \tilde{X} \ra X$ such that $L(\tilde{X}) = M$. By hypothesis, there is $P \in V(M)$, which corresponds to an $L$-rational map $\varphi: \tilde{X} \ra V$.  There is a nonempty Zariski-open subset $U \subset X$ such that: if $\tilde{U} = \pi^{-1}(U)$, then $\varphi|_{\tilde{U}}: \tilde{U} \ra V$ is a morphism and  $\pi|_{\tilde{U}}: \tilde{U} \ra U$ is a finite morphism.  By \cite[Lemma 12]{Clark07}, there is a $K$-rational zero-cycle on $U$ of degree $I(X)$. Then $\operatorname{Trace}_{L/K} \varphi_* \pi^* D$ 
is a divisor on $V$ of degree
\[ [M:L(X)] \cdot [L:K] \cdot I(X) = I(X) [M:L(X)][L(X):K(X)] = [M:K(X)]I(X) . \]
b) If $I(X) = 1$, then by part a) we have $I(V) \mid I(V_{/K(X)})$. Since for any extension $L/K$ we have $I(V_{/L}) \mid I(V)$, we conclude $I(V_{/K(X)}) = I(V)$.
\end{proof}
\noindent
Let $L/K$ be a finite extension.  Obviously $\widetilde{H}^1(L/K,A)$ can be nontrivial.  We give a sufficient condition for it to be finite.

\begin{lemma}
\label{THREELEMMA2}
\label{7.4}
Let $K$ be a WMW field, $L/K$ be a finite separable field extension, and $A_{/K}$ an abelian variety. Then $\widetilde{H}^1(L/K,A)$ is finite.
\end{lemma}
\begin{proof}
Let $M$ be the Galois closure of $L/K$.  Since $\widetilde{H}^1(L/K,A) \subset \widetilde{H}^1(M/K,A)$, we may replace $L$ with $M$ and thus assume that $L/K$ is finite Galois, say of 
degree $n$.  Because the period divides the index, we have $\widetilde{H}^1(L/K,A) = \widetilde{H}^1(L/K,A)[n]$. The short exact sequence of $K$-group schemes 
\[ 0 \ra A[n] \ra A \stackrel{[n]}{\ra} A \ra 0 \]
may be viewed as a short exact sequence of sheaves on the flat site of $\Spec K$, so we may take cohomology, getting the \textbf{flat Kummer sequence} 
\begin{equation} 
\label{FLATKUMMEREQ}
0 \ra A(K)/nA(K) \ra H^1(K,A[n]) \ra H^1(K,A)[n] \ra 0. 
\end{equation}
(The finite flat group scheme $A[n]$ is \'etale iff $\car K \nmid n$; in this case the Kummer sequence can be more simply viewed as a sequence of \'etale = Galois cohomology groups.  But in our application we need the general case.) There is also a Kummer sequence associated to multiplication by $n$ on $A_{/L}$. Restriction from $K$ to $L$ gives a commutative ladder 
\begin{equation}
\label{LADDEREQ}
\begin{CD}
0 @>>> A(K)/nA(K) @>>> H^1(K,A[n]) @>>> H^1(K,A)[n] @>>> 0 \\
 @.  @VVV @VVV @VVV  @. \\
0 @>>> A(L)/nA(L) @>>> H^1(L,A[n]) @>>> H^1(L,A)[n] @>>> 0 
\end{CD} 
\end{equation}
Let $\mathcal{K}$ and $\mathcal{C}$ be the kernel and cokernel of the restriction map $A(K)/nA(K) \ra A(L)/nA(L)$.  The Snake Lemma gives an exact sequence 
\[ 0 \ra \widetilde{H}^1(L/K,A[n])/\mathcal{K} \ra \widetilde{H}^1(L/K,A) \ra \mathcal{C}, \]
so to show that $\widetilde{H}^1(L/K,A)$ is finite it is sufficient to show that $\mathcal{C}$ and $\widetilde{H}^1(L/K,A[n])$ are both finite. The group $A(L)/nA(L)$ is finite because $K$ (hence $L$) is WMW, so its quotient $\mathcal{C}$ is also finite. Finally, because $L/K$ is Galois, by \cite[Thm., $\S$ 17.7]{Waterhouse}, $\widetilde{H}^1(L/K,A[n]) = H^1(\Aut(L/K),A[n](L))$. The right hand side of the last equation is the cohomology of a finite group with coefficients in a finite module, so it is a quotient of a finite group of cochains and thus is certainly finite.
\end{proof}

\begin{example}
\label{7.5}
Let $K = \F_p((t))$, and $K_n = K^{p^{-n}}$, so $[K_n:K] =p^n$.   Let $E_{/K}$ be a supersingular elliptic curve.  We will show that $H^1(K_2/K,E)$ is infinite.  \\
Step 1: We claim  $H^1(K,E)[p] \subset H^1(K_2/K,E)$.  Indeed, let 
$\eta \in H^1(K,E)[p]$, and using (\ref{FLATKUMMEREQ}) let $\xi$ be a lift of 
$\eta$ to $H^1(K,E[p])$.  Since $E[p]$ is a finite connected group scheme 
of Frobenius height $2$, by \cite[Prop. 76]{Shatz} we have $H^1(K_2/K,E[p]) = H^1(K,E[p])$.  Thus $\xi|_{K_2} = 0$, and using (\ref{LADDEREQ}) we deduce $\eta|_{K_2} = 0$.  \\
Step 2: By Theorem \ref{LOCALANALYTICTHM}b), $E(K)/pE(K)$ is infinite.  By the Local Duality Theorem, $H^1(K,E)[p]$ is infinite.  
Thus $H^1(K_2/K,E) \supset H^1(K,E)[p]$ is infinite. \\
Similarly, if $A_{/K}$ is any nontrivial abelian variety with $A[p](\overline{K}) = 0$, then if $n$ is the Frobenius height of $A[p]$ we have 
that $H^1(K_n/K,A)$ is infinite.
\end{example}  

\begin{remark}
In Lemma \ref{7.4} we assumed that $K$ is WMW and $L/K$ is separable.  In 
Example \ref{7.5} neither of these hypotheses holds.  It would be interesting to know whether either of the hypotheses of Lemma \ref{7.4} can be individually removed.
\end{remark}

\subsection{Proof of Theorem \ref{PARTTHREE}}
 \noindent
Let $t_1,\ldots,t_d$ be a separating transcendence basis for $L/K$ and put 
$K' =  K(t_1,\ldots,t_d)$.  By Lemma \ref{THREELEMMA1}, $H^1(K'/K,A) = 0$, so by Lemma \ref{LILEMMA}a), 
$\Res_{K'} S \subset H^1(K',A)[n]$ is infinite and $2$-LI over $\Z/n\Z$.  Moreover, since $K$ 
is WMW, by Proposition \ref{LANGNERONPROP} so is $K'$.  Thus we may as well 
assume that $K' = K$ and $L/K$ is a finite separable.  By Lemma \ref{THREELEMMA2}, $H^1(L/K,A)$ is finite.  By Lemma \ref{LILEMMA}b), there is $S' \subset S$ 
such that $\Res_L S' \subset H^1(L,A)[n]$ is infinite and $1$-LI over $\Z/n\Z$, and the latter condition means that every element of $\Res_L S'$ has period $n$.  Finally, we suppose that every element of $S$ has index $n$.  Then on the one hand every 
element of $\Res_L S'$ has index dividing $n$, whereas on the other hand every 
element of $\Res_L S'$ has period dividing its index, hence every element of 
$\Res_L S'$ has index $n$.

\end{document}